\documentclass[12pt]{amsart}
\usepackage[utf8]{inputenc}
\usepackage{amsmath}
\usepackage{amssymb}
\usepackage{amsfonts}
\usepackage{amsthm}
\usepackage{mathrsfs} 
\usepackage{bm}
\usepackage{bbm}
\usepackage{tikz}
\usepackage{centernot}
\usepackage{hyperref}
\usepackage[margin=1.2in]{geometry}
\DeclareMathOperator{\Vol}{Vol}

\newcommand{\E}{\Bbb{E}}
\newcommand{\Z}{\Bbb{Z}}
\newcommand{\R}{\Bbb{R}}
\newcommand{\G}{\Bbb{G}}

\renewcommand{\P}{\Bbb{P}}

\newcommand{\T}{\Bbb{T}}

\renewcommand{\AA}{\mathcal{A}}

\newcommand{\ep}{\epsilon}

\hypersetup{
    colorlinks=true,
    linkcolor=blue,
    filecolor=magenta,
    urlcolor=blue,
    citecolor=blue
}

\makeatletter
\newtheorem*{rep@theorem}{\rep@title}
\newcommand{\newreptheorem}[2]{%
\newenvironment{rep#1}[1]{%
 \def\rep@title{#2 \ref{##1}}%
 \begin{rep@theorem}}%
 {\end{rep@theorem}}}
\makeatother

\newtheorem{thm}{Theorem}
\newreptheorem{thm}{Theorem}

\newtheorem{result}{Result}[section]
\newtheorem{lem}[result]{Lemma}
\newtheorem{prp}[result]{Proposition}
\newtheorem{cor}[result]{Corollary}

\theoremstyle{definition}
\newtheorem{rmk}[result]{Remark}
\newtheorem*{defn}{Definition}

\newtheorem*{ack}{Acknowledgements}

\theoremstyle{remark}

\newcommand{\hide}[1]{}

\newcommand{\rough}[1]{}%\textbf{\textcolor{blue}{#1}}}
\definecolor{darkgreen}{RGB}{75,150,75}
\newcommand{\review}[1]{}%\textcolor{darkgreen}{#1}}

\newcommand{\hides}[1]{}%1}
\newcommand{\pub}[1]{}%\textcolor{purple}{#1}}

\title{Corner-free sets via the torus}
\author{Zach Hunter}
\email{zachary.hunter@exeter.ox.ac.uk}
\date{\today}

\begin{document}

\maketitle

\begin{abstract}
    A \textit{corner} is a triple of points in $\Z^2$ of the form $(x,y),(x+d,y),(x,y+d)$ where $d\neq 0$. One can think of them as being 2D-analogues to 3-term arithmetic progressions.
    
    In this short note, we extend ideas of Green-Wolf from this latter setting to the former, achieving slightly better constructions of corner-free sets.
\end{abstract}

\section{Introduction}

A (non-trivial) \textit{corner} is a set of three points in $\Z^2$ of the form $(x,y),(x+d,y),(x,y+d)$ for some $x,y,d\in \Z$ with $d\neq 0$. We say $A\subset \Z^2$ is \textit{corner-free} if it does not contain any non-trivial corners.

Let $r_{\angle}(N)$ denote the cardinality of the largest corner-free subset of the grid $[N]^2$ (where $[N]:=\{1,\dots,N\}$). For many years, the best known lower bound for $r_{\angle}(N)$ came from Behrend's construction and was of the form $N^2 2^{-(c+o(1))\sqrt{\log_2 N}}$ for $c = 2\sqrt{2}\approx 2.828\dots$.

Recently, improving upon work by Linial and Shraibman \cite{linial}, Green constructed corner-free sets of size $N^2 2^{-(c+o(1))\sqrt{\log_2 N}}$ where $c = 2\sqrt{2\log_2 \frac{4}{3}}\approx 1.822\dots$ \cite{green}.

We improve the lower order terms of this bound by using torus constructions, in a similar fashion to the work of \cite{wolf}.

\begin{thm}\label{main} Let $D,N$ be positive integers. There exists a corner-free subset $A \subset [N]^2$ where
\[|A| \gg \sqrt{D}(3/4)^DN^{2-2/D}\](here the implicit constant is independent of $D$).
\end{thm}
Consequently, plugging in $D = \left \lfloor \sqrt{\frac{\log_2 N}{\log_2(2/\sqrt{3})}}\right\rfloor $, we get:
\begin{cor}We have
\[r_\angle(N) \gg \log_2^{1/4} N\frac{N^2}{2^{2\sqrt{2\log_2\frac{4}{3}}\sqrt{\log_2 N}}}.\]
\end{cor}
\begin{rmk}Though the lower bound of $r_\angle(N)$ stated in \cite{green} hides lower-order terms with a $2^{-o(\sqrt{\log N})}$ factor, a careful analysis of their argument obtains \[r_\angle(N)\gg \log_2^{-1/4}N\frac{N^2}{2^{2\sqrt{2\log_2\frac{4}{3}}\sqrt{\log_2 N}}}.\] Thus, Theorem~\ref{main} improves things by a factor of $\log_2^{1/2} N$ (which is of identical shape to the improvement to $r_3(N)$ given by the aforementioned work \cite{wolf}).

\end{rmk}

\begin{ack}We thank Ben Green for informing us that this problem was of interest. We also thank Matt Kwan for comments which helped improve the exposition. Lastly we thank Fred Tyrrell for finding several typographical errors.

Some of this work was prepared at IST Austria, the author thanks them for their hospitality.

\end{ack}

\section{Preliminaries}
\subsection{Standard notation}

We shall use some standard asymptotic notation ($O,o,\gg$). Additionally, we will sometimes write $a\pm b$ to denote a quantity $x$ where $a-b\le x\le a+b$.

We will write $\T$ to denote the torus, $\R/\Z$. And similarly we write $\T^D$ to denote the $D$-dimensional torus, $\R^D/\Z^D$. We define the projection $\pi:\R\to \T;x\mapsto x+\Z$, and as an abuse of notation let $\pi^{-1}$ denote the inverse of $\pi$ restricted to $[0,1)$.

\subsection{Specialized definitions}\label{specialized}

It will be useful to think about tori with two distinct coordinates, so we can draw comparisons to the grid $[N]^2$. Thus we write $\G$ to denote $\T\times \T$ equipped with two coordinate maps $c_1,c_2:\G\to \R$, so that for $v = (a,b) \in \G$, we have
\[c_1(v) = \pi^{-1}(a),\quad c_2(v) = \pi^{-1}(b).\]\hide{For $\alpha \in \T$, we write $\alpha \cdot e_1$ to denote $(\alpha,0)\in \G$ and $\alpha \cdot e_2$ to denote $(0,\alpha) \in \G$.}

We now define $\psi:\G \to \R$ so that \[\psi(\theta) =c_1(\theta)+c_2(\theta).\] We also write $S$ to denote $\psi^{-1}([1/2,3/2))\subset \G$.

Lastly, we define $\phi:\G \to \R$ so that \[\phi(\theta) = c_1(\theta)-c_2(\theta),\] and furthermore extend $\phi$ to $\G^D\to \R^D$ coordinate-wise, so that $\phi(\theta_1,\dots,\theta_D) = (\phi(\theta_1),\dots,\phi(\theta_D))$.

\section{Grid lemmas}\label{gridlem}

We remind the reader to consult Section~\ref{specialized} for the definitions of $\G,\psi,S,\phi$.

We first need the following lemma.

\begin{lem}\label{psi iso}For any $\theta \in \G$, and any $\alpha,\beta \in \T$ such that $\{\theta + (\alpha,\beta),\theta+(\beta,\alpha)\} \subset S$, we have that $\psi(\theta+ (\alpha,\beta))=\psi(\theta+ (\beta,\alpha))$.
\begin{proof}First note that \[\pi(\psi(\theta+(\alpha,\beta))) = \pi(\psi(\theta))+\alpha +\beta= \pi(\psi(\theta+(\beta,\alpha))).\]Thus $\psi(\theta+(\alpha,\beta))-\psi(\theta+(\beta,\alpha)) \in \Z$. 

Recalling the assumption that \[\psi(\{\theta+(\alpha,\beta),\theta+(\beta,\alpha)\})\subset \psi(S) \subset [1/2,3/2),\] we see that $\psi(\theta+(\alpha,\beta))-\psi(\theta+(\beta,\alpha)) \in (-1,1)$. 

Combining these two observations, $\psi(\theta+(\alpha,\beta))-\psi(\theta+(\beta,\alpha))$ must equal $0$, as desired.\end{proof}
\end{lem}
We can now deduce that addition into $S$ ``behaves nicely'' with respect to $\phi$. 
\begin{lem}\label{phi basic} Consider $\theta \in \G$ and $\alpha\in \T$ satisfying $\{\theta+(\alpha,0),\theta+(0,\alpha)\}\subset S$. Then \[\phi(\theta+(\alpha,0)) -\phi(\theta)=\phi(\theta)- \phi(\theta+(0,\alpha)).\]
\begin{rmk}To illustrate what Lemma~\ref{phi basic} is saying, we consider an example. Taking $\theta = (\pi(3/4),\pi(1/4)) \in \G$ and $\alpha = \pi(x)$ for some $x\in (1/4,3/4)$, we have 
\[\phi(\theta +(\alpha,0)) -\phi(\theta) = (x-1/4)-(3/4-1/4)= x-3/4,\]
\[\phi(\theta)-\phi(\theta+(0,\alpha)) = (3/4-1/4)-(3/4-1/4-x)= x,\]which aren't equal. This is due to the fact that adding $\alpha$ to the first coordinate makes us ``wrap around'', but this doesn't happen for the second coordinate.  

We wish to avoid this, so that we can later make use of the geometric insight that in $\R^D$ (where things don't wrap around), lines intersect spheres in at most two points.
\end{rmk}

\begin{proof}By Lemma~\ref{psi iso} (with $\beta = 0$), we have that $\psi(\theta+(\alpha,0)) = \psi(\theta+(0,\alpha))$. Hence with $\Delta_1 := c_1(\theta+(\alpha,0))-c_1(\theta+(0,\alpha)), \Delta_2 := c_2(\theta+(\alpha,0))-c_2(\theta+(0,\alpha))$, we get
\begin{align*}
    0 &= \psi(\theta+(\alpha,0)) - \psi(\theta+(0,\alpha))\\
    &= \Delta_1+\Delta_2\\
    &\implies \Delta_1 = -\Delta_2.\\
\end{align*} 

We conclude by noting 
\[\Delta_1 +\phi(\theta) = \phi(\theta+(\alpha,0))\]and
\[\phi(\theta)+\Delta_2 = \phi(\theta+(0,\alpha))\](here we used the facts that $c_1(\theta+(0,\alpha)) = c_1(\theta)$ and $c_2(\theta+(\alpha,0)) = c_2(\theta)$). The result follows from some minor rearranging.
\end{proof}
\end{lem}

Inspecting the proof above, we obtain the following corollary.
\begin{cor} \label{phi iso} Consider $\theta \in \G$ and $\alpha \in \T$ with $\{\theta+(\alpha,0),\theta+(0,\alpha)\} \subset S$. Then $\phi(\theta+(\alpha,0))-\phi(\theta) = \phi(\theta)-\phi(\theta+(0,\alpha)) = \Delta$, where $\Delta:=c_1(\theta+(\alpha,0))-c_1(\theta)$. 

Notably, $\pi(\Delta) = \alpha$ and thus $||\alpha||_\T \le |\Delta|$.
\end{cor}

\section{Construction}

\subsection{Setup and motivation}

In this section, we shall prove Theorem~\ref{main}. This is done by considering a construction depending on several parameters ($r,\delta,\theta,\mu$, described below), and then optimizing them with respect to a given $N,D$. 

For later reference, we now collect the relevant definitions of our construction. Afterwards, we will conclude this subsection by commenting on their meaning.

\begin{defn}\label{convars} Recall the definitions of $S,\phi,\G$ from Section~\ref{specialized}. 
Given a dimension $D$ and $r,\delta > 0$, we define the following subset of the $D$-dimensional torus grid $\G^D$,
\[S_{r,\delta;D}:= \{\theta \in S^D : ||\phi(\theta)||_2 \in [r-\delta,r)\}.\]
\noindent Then, given $\theta  \in \T^D,\mu \in \G^D$, we define the function \[f=f_{\theta,\mu}:\Z^2\to \G^D; (x,y)\mapsto ((x\theta_1, y\theta_1)+\mu_1,\dots,(x\theta_D,y\theta_D)+\mu_D) .\]Lastly we define the set $A = A_{r,\delta;D;\theta,\mu} := \{(x,y)\in [N]^2:f(x,y)\in S_{r,\delta;D}\}$.
\end{defn}
\noindent Theorem~\ref{main} shall be obtained by finding $r,\delta,\theta$ such that $A$ is corner-free for all choices of $\mu \in \G^D$, with $\E_\mu[|A|]$ being sufficiently large.

For comparison, we briefly recall the construction of $3$-AP-free\footnote{We refer to $3$-term arithmetic progressions as $3$-AP's.} sets by Green-Wolf \cite{wolf}. Green-Wolf considered a random affine homomorphism $g=g_{\theta,\mu}:\Z \to \T^D;n\mapsto n\theta +\mu$ (here $\theta,\mu\in \T^D$). For a $3$-AP $P\subset [N] \cap g^{-1}(\pi([0,1/2)^D))$, one has that $\pi^{-1}(g(P))$ maps to a set $\tilde{P}$ of three collinear points in $\R^D$ (this is due to a standard ``Freiman isomorphism'' argument). By fixing a thin annulus $\AA \subset [0,1/2)^D$, and taking $n\in [N]$ such that $\pi^{-1}(g(\theta n)) \in \AA $, we get our large $3$-AP-free set. 

Our construction works quite similarly. We now use the random affine homomorphism $f:\Z^2 \to \G^D$. For any corner $C\subset [N]^2 \cap f^{-1}(S^D)$, we will have that $\phi(f(C))$ maps to a set $\tilde{P}$ of three colinear points in $\R^D$ (this is now due to the arguments from Section~\ref{gridlem}). So again, we will fix a thin annulus $\AA\subset \R^D$ and obtain a large corner-free set by taking the $v\in [N]^2 \cap f^{-1}(S^D)$ where $\phi(f(v)) \in \AA$. 

The only real difference is that instead of getting an ``approximate homomorphism'' from $\T^D$ to $\R^D$ by ``pulling back'' $\pi^{-1}$ and restricting to $\pi([0,1/2)^D)$, we now use the map $\phi:\G^D \to \R^D$ and restrict to $S^D$ for our approximate homomorphism. This does better, because $S^D$ has greater volume than $\pi([0,1/2)^D)$.

\subsection{Proofs}We remind the reader to consult Definition~\ref{convars} for the definition of the objects $f,A,S_{\cdot,\cdot;\cdot}$.

It remains to deduce Theorem~\ref{main}. We first obtain the following. 
\begin{lem}\label{smalld} Let $x,y,|d|\in [N]$ be such that $\{f(x,y),f(x+d,y),f(x,y+d)\} \subset S_{r,\delta;D}$. Then $\sum_{i=1}^D ||d\theta_i||_\T^2 \le 2r\delta$.
\begin{proof} By Lemma~\ref{phi basic}, there exists $\Delta = \phi\circ f(x+d,y)-\phi\circ f(x,y) = \phi\circ f(x,y)-\phi\circ f(x,y+d)$ in $\R^D$, and by Corollary~\ref{phi iso} we have that $\sum_{i=1}^D ||d\theta_i||_\T^2 \le ||\Delta||_2^2$.

By parallelogram law,
\[2||\phi\circ f(x,y)||_2^2 +2||\Delta||_2^2 = ||\phi\circ f(x,y)+\Delta||_2^2 +||\phi\circ f(x,y)-\Delta||_2^2,\]
\[\implies 2||\Delta||_2^2 \le 4r\delta. \]\end{proof}
\end{lem}
Let $B_0\subset \R^D$ be the ball around the origin with Euclidean radius $\sqrt{2r\delta}$. Let $B = \pi(B_0)$. Using Lemma~\ref{smalld}, we can now get the following.

\begin{cor}\label{size}Suppose $\theta \in \G^D$ is such that $d\theta \not \in B $ for all $d \in [N]$. 

Then there exists a choice of $\mu \in \G^D$ such that with $A = A_{r,\delta;D;\theta,\mu}$, we have
\[A \subset [N]^2 \textrm{ is corner-free}\]
\[|A| \ge N^2 \Vol(S_{r,\delta;D}).\]
\begin{proof}We note that $-d\theta \in B$ if and only if $d\theta \in B$, hence our assumption implies $d\theta \not\in B$ whenever $|d|\in [N]$. Due to Lemma~\ref{smalld}, we have that $A_{r,\delta;D;\theta,\mu}$ will be corner-free for every choice of $\mu \in \G^D$.

We now choose $\mu \in \G^D$ randomly. For each $(x,y) \in [N]^2$, we see that $\P_\mu(f(x,y) \in S_{r,\delta;D}) = \Vol(S_{r,\delta;D})$. It follows that
\[\E_\mu[|A_{r,\delta;D;\theta,\mu}|] = \sum_{(x,y)\in [N]^2} \P_\mu(f(x,y)\in S_{r,\delta;D}) = N^2\Vol(S_{r,\delta;D}).\]By the probabalistic method, we conclude there is some choice of $\mu \in \G^D$ where $|A_{r,\delta;D;\theta,\mu}|$ is at least the RHS, which gives the desired result.\end{proof}
\end{cor}

We shall conclude by choosing our parameters so $\Vol(S_{r,\delta;D})$ is large while $\Vol(B)$ is (sufficiently) small, allowing us to use Corollary~\ref{size}.

\begin{prp}\label{pigeon}There exists an absolute constant $c^*>0$ so that the following holds. For each $D,\delta$, there exists $r$ such that $\Vol(S_{r,\delta;D}) \ge c^*\delta(3/4)^D$.
\begin{proof} Let $m^2 = \E_{\mu\sim \G}[ ||\phi(\mu)||_2^2\, | \,\mu \in S] = \frac{5}{24} >0$.

By Hoeffding's inequality (a standard concentration result, see \cite[Theorem~2]{hoeffding}), we have that $\P_{\mu\sim S^D} [ |Dm^2 -||\phi(\mu)||_2^2| >D^{1/2}] \le 2\exp(-2) = 1-\ep$ for some $\ep>0$. Hence, conditioned on $\mu \in S^D$, we have $||\phi(\mu)||_2^2 = m^2D \pm \sqrt{D}$ or equivalently $||\phi(\mu)||_2 = m\sqrt{D} \pm K$ with positive probability $\ep$ (here $K = O(1)$). Chopping this error into $\lfloor K\delta^{-1}\rfloor$ intervals of length $\delta$, $||\phi(\mu)||_2$ lands in one of these intervals with probability $\ge \frac{\ep}{2K}\delta\gg \delta$ by pigeonhole (so we take $c^* = \ep/2K$). 

The result follows as $\Vol(S_{r,\delta;D}) = (3/4)^D \P_{\mu\sim S^D}(||\phi(\mu)||_2 \in [r-\delta,r))$.\end{proof}
\end{prp}
Lastly, we note that $\Vol(B)\le \Vol(B_0) $ (the volume of the Euclidean ball in $\R^D$ with Euclidean radius $\sqrt{2r\delta}$), and $\Vol(B_0) \le \left(\frac{O(1)}{D}r\delta\right)^{D/2}$. Furthermore, given $r\le \sqrt{D}$ (which we may assume WLOG as $||\phi(\G^D)||_2$ is supported on $[0,\sqrt{D}]$), the above simplifies to $\Vol(B)\le \left(\frac{O(\delta)}{\sqrt{D}}\right)^{D/2}$.

\begin{proof}[Proof of Theorem~\ref{main}] We simply apply Corollary~\ref{size} for an appropriate choice of parameters.

In particular, we take $\delta = c\sqrt{D}N^{-2/D}$ for some constant $c>0$ and $r$ according to Proposition~\ref{pigeon}, so that $\Vol(S_{r,\delta;D})\ge c^* \delta (3/4)^D$. We have that $N\Vol(B) = O(c)^{D/2}$, thus for sufficiently small $c$ (with respect to the implicit constant), we have $N\Vol(B)<1$. So, choosing $\theta\in \T^D$ uniformly at random, we have with positive probability $\theta d \not\in B$ for all $d\in [N]$ by a union bound.

Thus we get a corner-free set $A \subset [N]^2$ with \[|A| \ge N^2\Vol(S_{r,\delta;D}) \ge N^{2-2/D} (3/4)^D(cc^* \sqrt{D}) \gg N^{2-2/D}(3/4)^D\sqrt{D}.\]\end{proof}

{}

\begin{thebibliography}{}
\bibitem{green} B. Green, \textit{Lower bounds for corner-free sets,} in \textit{New Zealand Journal of Mathematics} \textbf{51} (2021).

\bibitem{wolf} B. Green and J. Wolf, \textit{A note on Elkin's improvement of Behrend's construction,} in  \textit{Additive Number Theory}, 141–144, Springer, New York 2010.

\bibitem{hoeffding} W. Hoeffding, \textit{Probability inequalities for sums of bounded random variables,} in \textit{Journal of the American Statistical Association} \textbf{58} (1963), p. 13–30.


\bibitem{linial} N. Linial and A. Shraibman, \textit{Larger Corner-Free Sets from Better NOF Exactly-$N$ Protocols,} in \textit{Discrete Analysis} \textbf{19} (2021).

\end{thebibliography}
\end{document}